\documentclass[10pt,twoside,leqno]{amsart}

 \newtheorem{thm}{Theorem}[section]
 \newtheorem{cor}[thm]{Corollary}
 \newtheorem{lem}[thm]{Lemma}
 \newtheorem{prop}[thm]{Proposition}
 \theoremstyle{definition}
 \newtheorem{defn}[thm]{Definition}
 \theoremstyle{remark}
 \newtheorem{rem}[thm]{Remark}
 \newtheorem*{ex}{Example}
 \numberwithin{equation}{section}

\newcommand{\be}{\begin{equation}}
\newcommand{\ee}{\end{equation}}
\newcommand{\beq}{\begin{equation}}
\newcommand{\enq}{\end{equation}}

\newcommand{\Wc}{\mathcal{W}}
\newcommand{\At}{ {\tilde{A}}}
\newcommand{\Bt}{ {\tilde{B}}}
\newcommand{\Mt}{ {\tilde{M}}}
\newcommand{\Hc}{ {\mathcal{H}}}
\newcommand{\Ran}{ {\mbox{\rm Ran}}}
\newcommand{\K}{ {\mathcal{K}}}
\newcommand{\Sc}{ {\mathcal{S}}}
\newcommand{\Tc}{ {\mathcal{T}}}
\newcommand{\cH}{ {\mathcal{H}}}
\newcommand{\cK}{ {\mathcal{K}}}
\newcommand{\cL}{ {\mathcal{L}}}
\newcommand{\Gammat}{ {\tilde{\Gamma}}}
\newcommand{\sign}{\mbox{\rm sign}}
\newcommand{\Rr}{{\mathbb{R}}}
\newcommand{\Nn}{{\mathbb{N}}}
\newcommand{\Cc}{{\mathbb{C}}}

\newcommand{\llangle}{\left\langle}
\newcommand{\rrangle}{\right\rangle}

\newcommand{\A}{A}

\newcommand{\essran}{\mbox{\rm essran}}
\begin{document}
\title[Abstract $M$-functions]{The abstract Titchmarsh-Weyl $M$-function for adjoint operator pairs and its relation to the spectrum
}

\author[Brown]{Malcolm Brown}
\address{School of Computer Science, Cardiff University, Queen's Buildings, 5 The Parade, Cardiff CF24 3AA, UK}
\email{Malcolm.Brown@cs.cardiff.ac.uk}

\author[Hinchcliffe]{James Hinchcliffe}
\address{School of Mathematics, Cardiff University, Senghennydd Road, Cardiff CF24 4AG, UK}
\email{hinchcliffeje@cardiff.ac.uk}

\author[Marletta]{Marco Marletta}
\address{School of Mathematics, Cardiff University, Senghennydd Road, Cardiff CF24 4AG, UK}
\email{MarlettaM@cardiff.ac.uk}

\author[Naboko]{Serguei Naboko}
\address{Department of Math.~Physics, Institute of Physics, St.~Petersburg State University, 1 Ulianovskaia, St.~Petergoff, St.~Petersburg, 198504, Russia}
\email{naboko@snoopy.phys.spbu.ru}

\author[Wood]{Ian Wood}
\address{Institute of Mathematics and Physics, Aberystwyth University, Penglais, Aberystwyth, Ceredigion SY 23 3BZ, UK}
\email{ian.wood@aber.ac.uk}

\begin{abstract}
In the setting of  adjoint pairs of operators we consider the question: to what extent does the 
Weyl $M$-function see the same singularities as the resolvent of a certain restriction $A_B$ of the maximal operator? We obtain results showing that it is possible to describe explicitly certain spaces $\Sc$ and $\tilde{\Sc}$ such that the  resolvent bordered by projections onto these subspaces is analytic everywhere that the $M$-function is analytic.  
We present three examples -- one involving a Hain-L\"{u}st type operator,
 one involving a perturbed Friedrichs operator and one involving a simple  ordinary
 differential operators on a half line -- which together indicate that the abstract results
 are probably best possible. 
\end{abstract}

\thanks{James Hinchcliffe and Serguei Naboko wish to thank the British EPSRC for financial support under 
grant EP/C008324/1 "Spectral Problems on Families of Domains and Operator M-functions". Serguei Naboko wishes to thank the
Russian RFBR for grant 06-01-00219. All authors wish to thank INTAS for financial support under INTAS Project No. 051000008-7883.}

\maketitle
\noindent{\bf AMS(MOS) Subject Classifications: 35J25, 35P05, 47A10, 47A11}
\section{Introduction}\label{section:1}
 In the theory of inverse problems for Schr\"{o}dinger operators on a half
 line,
\be -y'' + q(x)y = \lambda y, \;\;\; x\in (0,\infty), \label{intro1} \ee
 it has been well known since the work of Borg \cite{Borg49},
 of Marchenko \cite{Marchenko} and of Gelfand and Levitan \cite{Gelfand} that the function
 $q$ is uniquely determined by the Titchmarsh-Weyl function for the problem. Here $q$ is assumed
 to be real valued and integrable over any finite sub-interval of $[0,\infty)$ and to give rise 
 to a so-called limit point case at infinity: that is, one requires only a boundary condition at 
 the origin, and no boundary condition at infinity, in order to obtain a selfadjoint operator associated with
 the expression on the left hand side of (\ref{intro1}).
 
 The Titchmarsh Weyl function $M(\lambda)$ for this problem can be regarded as a Dirichlet
 to Neumann map for the problem. Suppose that we define a `maximal' operator $A^*$ by
\[ D(A^*) = \{ y \in L^2(0,\infty) \; | \; -y'' + q y \in L^2(0,\infty) \}, \]
\[ A^* y = -y'' + q y, \]
 where $y''$ is to be understood in the sense of weak derivatives; also define some
 `boundary' operators $\Gamma_1$ and $\Gamma_2$ on $D(A^*)$ by
\[ \Gamma_1 y = y(0), \;\;\; \Gamma_2 y = -y'(0). \]
 Then the Titchmarsh Weyl function may be defined by the expression
\[ M(\lambda) = \Gamma_2\left(\left. \Gamma_1\right|_{\ker(A^*-\lambda I)} \right)^{-1}, \]
or equivalently
\[ M(\lambda)y(0) = -y'(0) \;\;\; \mbox{when $-y'' + qy = \lambda y$ and $y\in L^2(0,\infty)$}. \]
 If we let $A_D$ denote the `Dirichlet restriction' of $A^*$, that is the restriction of $A^*$ to
\[ D(A_D) = D(A^*)\cap\ker(\Gamma_1), \]
 then the $M$-function is easily seen to be well defined for $\lambda \not\in \sigma(A_D)$.
 One may show that $(A_D-\lambda)^{-1}$ has the same poles as $M(\lambda)$, and the
 famous Weyl Kodaira formula relates the spectral measure $\rho$ of 
 $A_D$ to $M$:
\[ d\rho(k) = \frac{1}{\pi}{w-{\rm lim}}_{\epsilon\searrow 0} \Im M(k + i\epsilon)dk. \]
 In short, complete information about the original operator is encoded in $M$.

 For PDEs, similar inverse results are also available. For Schr\"{o}dinger operators on smooth domains
 with smooth potentials, for instance, the Dirichlet-to-Neumann map $M(\lambda)$ determines the potential uniquely. 
 Moreover in this PDE case it is not necessary to know $M(\lambda)$ as a function of 
 $\lambda$: it suffices to know it for one value of $\lambda$ for which it is well defined. 
 For more general classes of PDEs there are many results guaranteeing that the coefficients can be recovered 
 up to some explicit transformations. See  Isakov \cite{Isakov} for a review of inverse problems for elliptic PDEs.  
 
 In this paper we consider similar questions in the totally abstract setting of boundary triples (cf. Section \ref{section:2} for the definition).
As shown in the papers by Kre\u{\i}n, Langer and Textorius \cite{KL73,KL77,LT77} on
extensions of symmetric operators,
under an assumption of complete nonselfadjointness
of the underlying symmetric minimal operator, the
maximal operator is determined up to unitary equivalence by the 
M-function.
Moreover, recently Ryzhov \cite{Ryzhov} has shown that under the same assumptions
and an additional invertibility condition imposed on the Dirichlet 
restriction $A_D$,
the operators $A_D$ and $\Gamma_2 A_D^{-1}$ are determined
by the difference $M(z) - M(0)$ up to unitary equivalence.

 For the non-symmetric case, the
 authors considered in \cite{kn:BMNW} the question of behaviour of the abstract $M$-function(s) 
 near the boundary of the essential spectrum and asked: to what extent does the $M$-function see
 the same singularities as the resolvent of a certain restriction $A_B$ of the maximal operator? 

 In this paper we obtain results showing that it is possible to describe explicitly certain
 spaces $\Sc$ and $\tilde{\Sc}$ such that the bordered resolvent $P_{\tilde{\Sc}}(A_B-\lambda I)^{-1}
 P_{\Sc}$, in which the $P$ are orthogonal projections onto the spaces indicated, is analytic everywhere
 that $M(\lambda)$ is analytic. The spaces $\Sc$ and $\tilde{\Sc}$ are, in general,
 not closed. However we present three examples -- one involving a Hain-L\"{u}st type operator,
 one involving a perturbed Friedrichs operator and one involving  simple ordinary
 differential operators on a half line -- which together indicate that the abstract results
 in Section \ref{section:3} are probably best possible. As a result we conclude that the
 abstract approach to inverse problems may yield rather limited results unless further hypotheses
 are introduced which reflect properties of problems involving concrete ordinary and partial
 differential expressions.

We should mention that since their introduction by Vishik \cite{Vis52} for second order elliptic operators and Lyantze and Storozh \cite{Lyantze} for adjoint pairs of  abstract operators, boundary triplets have been widely used to characterise extensions of operators and investigate spectral properties using Weyl M-functions. An extension of the theory to relations can be found in the work of  
Malamud and Mogilevskii \cite{MM99,MM02}. For related recent results, particularly in the context of PDEs, we refer to the works of Alpay and Behrndt \cite{AB08}, Behrndt and Langer \cite{BL07}, Brown, Grubb, Wood \cite{BGW08}, Gesztesy and Mitrea \cite{GMZ07,GM08a,GM08b} and also to Posilicano \cite{Pos07,Pos08} and Post \cite{Post08}. 

The authors wish to thank the referee for many helpful comments.

\section{Background theory of boundary triples and Weyl functions}\label{section:2}
Throughout this article we will make the following assumptions:
\begin{enumerate}
  \item $A$, $\At$ are closed, densely defined operators in a Hilbert space $H$.
  \item $A$ and $\At$ are an adjoint pair, i.e. $A^*\supseteq\At$ and $\At^*\supseteq A$.
\end{enumerate}

\begin{prop}\cite[(Lyantze, Storozh '83)]{Lyantze}. For each adjoint pair of closed densely defined operators on $H$,
there exist ``boundary spaces'' $\cH$, $\cK$ and ``boundary operators''
  \[ \Gamma_1:D(\At^*)\to\cH,\quad \Gamma_2:D(\At^*)\to\cK,\quad \Gammat_1:D(A^*)\to\cK\quad \hbox{ and }\quad  \Gammat_2:D(A^*)\to\cH \]
  such that for $u\in D(\At^*) $ and $v\in D(A^*)$ we have an abstract Green formula
        \begin{equation}\label{Green}
          (\At^* u, v)_H - (u,A^*v)_H = (\Gamma_1 u, \Gammat_2 v)_\cH - (\Gamma_2 u, \Gammat_1v)_\cK.
        \end{equation} 
The boundary operators $\Gamma_1$, $\Gamma_2$, $\Gammat_1$ and $  \Gammat_2 $ are bounded with respect to the graph norm. The pair $(\Gamma_1,\Gamma_2)$ is 
surjective onto $\cH\times\cK$ and $(\Gammat_1,\Gammat_2)$ is surjective onto $\cK\times\cH$. Moreover, we have 
\begin{equation}\label{domains}
	D(A)= D(\At^*)\cap\ker\Gamma_1\cap \ker\Gamma_2 \quad \hbox{ and } \quad D(\At)= D(A^*)\cap\ker\Gammat_1\cap \ker\Gammat_2.
\end{equation}
The collection $\{\cH\oplus\cK, (\Gamma_1,\Gamma_2), (\Gammat_1,\Gammat_2)\}$ is called a boundary triplet for the adjoint pair $A,\At$.
\end{prop}

Malamud and Mogilevskii \cite{MM99,MM02} use this setting to define Weyl $M$-functions and $\gamma$-fields associated with boundary 
triplets and to obtain Kre\u{\i}n formulae for the resolvents. We now summarize some results, using however a slightly
different setting taken from \cite{kn:BMNW} in which the boundary conditions and Weyl function contain an additional operator $B\in\cL(\cK,\cH)$.

\begin{defn}\label{defmfn} Let $B\in\cL(\cK,\cH)$ and $\Bt\in\cL(\cH,\cK)$. We define extensions
of $A$ and $\At$ (respectively) by
\[ A_B:=\At^*\vert_{\ker(\Gamma_1-B\Gamma_2)} \hbox{ and } \At_\Bt:=A^*\vert_{\ker(\Gammat_1-\Bt\Gammat_2)}.\]  
In the following, we assume $\rho(A_B)\neq\emptyset$, in particular $A_B$ is a closed operator.
For $\lambda\in\rho(A_B)$, we define the $M$-function via
\[ M_B(\lambda):\Ran(\Gamma_1-B\Gamma_2)\to\cK,\ M_B(\lambda)(\Gamma_1-B\Gamma_2) u=\Gamma_2 u \hbox{ for all } u\in \ker(\At^*-\lambda)\]
and for $\lambda\in\rho(\At_\Bt)$, we define 
\[ \Mt_\Bt(\lambda):\Ran(\Gammat_1-\Bt\Gammat_2)\to\cH,\ \Mt_\Bt(\lambda)(\Gammat_1-\Bt\Gammat_2) v=\Gammat_2 v 
\hbox{ for all } v\in \ker(A^*-\lambda).\]
\end{defn}
It is easy to prove that $M_B(\lambda)$ and $\Mt_\Bt(\lambda)$ are well defined for $\lambda\in \rho(A_B)$ and $\lambda\in \rho(\At_\Bt)$
respectively.
\begin{defn} (Solution Operator) For $\lambda\in\rho(A_B)$, we define the linear operator $S_{\lambda,B}:\Ran(\Gamma_1-B\Gamma_2)\to 
\ker(\At^*-\lambda)$ by
\begin{eqnarray}\label{slamdef}
  (\At^*-\lambda)S_{\lambda,B} f=0,\ (\Gamma_1-B\Gamma_2)S_{\lambda,B} f=f,
\end{eqnarray}
i.e.  $S_{\lambda,B}=\left( (\Gamma_1-B\Gamma_2)\vert_{\ker(\At^*-\lambda)}\right)^{-1}$.
\end{defn}
Since we shall use solution operators quite extensively in the sequel, we include the proof of the following lemma, for completeness.
\begin{lem}\label{slamwd}
$S_{\lambda,B}$ is well-defined for $\lambda\in\rho(A_B)$. Moreover for each
$f\in\Ran(\Gamma_1-B\Gamma_2)$ the map from $\rho(A_B)\to H$ given by $\lambda\mapsto S_{\lambda,B} f$ is analytic.
\end{lem}

\begin{proof}
For $f\in \Ran(\Gamma_1-B\Gamma_2)$, choose any $w\in D(\At^*)$ such that $(\Gamma_1-B\Gamma_2)w=f$. Let 
$v=-(A_B-\lambda)^{-1}(\At^*-\lambda)w$. Then $v+w\in\ker(\At^*-\lambda)$ and $(\Gamma_1-B\Gamma_2)(v+w)=(\Gamma_1-B\Gamma_2)w=f$, 
so a solution to (\ref{slamdef}) exists and is given by 
\begin{equation}
  S_{\lambda,B} f= \left( I-(A_B-\lambda)^{-1}(\At^*-\lambda)\right)w \label{eq:eff}
\end{equation}
for any $w\in D(\At^*)$ such that $(\Gamma_1-B\Gamma_2)w=f$. Moreover $S_{\lambda,B} f$ is well defined because the solution to 
(\ref{slamdef}) is unique. For suppose $u_1$ and $u_2$ are two solutions. Then $(u_1-u_2)\in\ker(\At^*-\lambda)\cap 
\ker(\Gamma_1-B\Gamma_2)$, so $u_1-u_2\in D(A_B)$ and $(A_B-\lambda)(u_1-u_2)=0$. As $\lambda\in\rho(A_B)$, $u_1=u_2$. 
The analyticity of $S_{\lambda,B}$ as a function of $\lambda$ is immediate from (\ref{eq:eff}) using the fact that the 
choice of $w$ does not depend on $\lambda$.
\end{proof}

\begin{cor} \label{corollary:1} Under the hypotheses of Lemma \ref{slamwd},
\be S_{\lambda,B} =  S_{\lambda_0,B} + (\lambda-\lambda_0)(A_B-\lambda)^{-1} S_{\lambda_0,B}. \label{eq:sdiff} \ee
\end{cor}
\begin{proof}
Fix $\lambda_0\in\rho(A_B)$ and choose $w=S_{\lambda_0,B}f$. Then
\begin{eqnarray*}
  S_{\lambda,B} f & = & \left( S_{\lambda_0,B}-(A_B-\lambda)^{-1}(\At^*-\lambda)S_{\lambda_0,B}\right) f \\
            & = & S_{\lambda_0,B}f + (\lambda-\lambda_0)(A_B-\lambda)^{-1}S_{\lambda_0,B}f. \nonumber
\end{eqnarray*}
\vskip-1em
\end{proof}
Note that the identity (\ref{eq:sdiff}) may be regarded as a Hilbert identity for the difference of resolvents
corresponding to different boundary conditions.

To be able to study spectral properties of the operator $A_B$ via the $M$-function, we need to relate the $M$-function to the resolvent. This can be done in the following way:

\begin{thm}
\begin{enumerate}
 \item Let $\lambda,\lambda_0\in\rho(A_B)$. Then on $\Ran(\Gamma_1-B\Gamma_2)$
\begin{eqnarray*}
M_B(\lambda) &=& \Gamma_2\left( I+(\lambda-\lambda_0)(A_B-\lambda)^{-1}\right)S_{\lambda_0,B}\\ 
&=& \Gamma_2(A_B-\lambda_0)(A_B-\lambda)^{-1}S_{\lambda_0,B}.
\end{eqnarray*}
	\item Let $B,C\in \cL(\cK,\cH)$, $\lambda\in\rho(A_B)\cap\rho(A_C)$. Then 
\begin{eqnarray*}\label{Krein}
(A_B-\lambda)^{-1} &=& (A_C-\lambda)^{-1}-S_{\lambda,C}(I+(B-C)M_B(\lambda))(\Gamma_1-B\Gamma_2)(A_C-\lambda)^{-1}\\ \nonumber
                   &=& (A_C-\lambda)^{-1}-S_{\lambda,C}(I+(B-C)M_B(\lambda))(C-B)\Gamma_2(A_C-\lambda)^{-1}.
\end{eqnarray*}
\end{enumerate}
\end{thm}

\begin{proof} Part (1) is just Proposition 4.6 from \cite{kn:BMNW}, while part (2) is a slight improvement to Theorem 4.7 of the same paper. We include the proof of (2) for completeness:
Let $u\in H$. Set $v:=\left((A_B-\lambda)^{-1}-(A_C-\lambda)^{-1}\right)u$. Since $v\in\ker(\At^*-\lambda)$, we have 
$M_B(\lambda)(\Gamma_1-B\Gamma_2)v=\Gamma_2 v$. Then
\begin{eqnarray*}
\left(\Gamma_1-C\Gamma_2\right)v &=& \left[\Gamma_1-B\Gamma_2+(B-C)M_B(\lambda)(\Gamma_1-B\Gamma_2)\right] v\\ \nonumber
                          &=& (I+(B-C)M_B(\lambda))(\Gamma_1-B\Gamma_2) v\\ \nonumber
                          &=& -(I+(B-C)M_B(\lambda))(\Gamma_1-B\Gamma_2) (A_C-\lambda)^{-1}u.
\end{eqnarray*}
Set $f:=-(I+(B-C)M_B(\lambda))(\Gamma_1-B\Gamma_2) (A_C-\lambda)^{-1}u$. Then by the above calculation,  $f\in\Ran(\Gamma_1-C\Gamma_2)$
and $S_{\lambda, C}f=v=\left((A_B-\lambda)^{-1}-(A_C-\lambda)^{-1}\right)u$. Therefore, 
\begin{eqnarray*}
(A_B-\lambda)^{-1}&=& (A_C-\lambda)^{-1}-S_{\lambda, C}(I+(B-C)M_B(\lambda))(\Gamma_1-B\Gamma_2) (A_C-\lambda)^{-1}.	
\end{eqnarray*}
\end{proof}

\section{How much of an operator can its Weyl function determine?}\label{section:3}
In this section we wish to know how much of the spectrum of an operator
can be seen by its Weyl function. In the symmetric case, complete non-selfadjointness of the minimal operator $A$ is required to recover the operator (up to unitary equivalence) from the Weyl function (see e.g.~\cite{Ryzhov}). Motivated by this, we fix $\mu_0\not\in \sigma(A_B)$ and define the spaces
\be \Sc = \mbox{\rm Span}_{\delta\not\in \sigma(A_B)}(A_B-\delta I)^{-1}
        \mbox{Ran}(S_{\mu_0,B})\label{eq:mm1}, \ee
\be \Tc = \mbox{\rm Span}_{\mu\not\in \sigma(A_B)}
        \mbox{Ran}(S_{\mu,B})\label{eq:mm1b}, \ee
where $S_{\mu,B} = \left(\left.(\Gamma_1-B\Gamma_2)\right|_{\ker(\At^*-\mu I)}\right)^{-1}$
is the solution operator. Here $\mbox{\rm Span}$ denotes the set of finite linear combinations of 
vectors from the sets indicated.

The spaces $\Sc$ depend on the choice of $\mu_0$, but this dependence will not be indicated
explicitly. Moreover the closures of $\Sc$ do not depend on $\mu_0$, as the following lemma
shows.
\begin{lem}\label{lemma:mm0} Suppose that there exists a sequence $(z_n)$ in $\Cc$ which tends
to infinity and is such that the family of operators $(z_n(A_B-z_nI)^{-1})_{n\in\Nn}$ is 
bounded. Then 
\[ \overline{\Sc} = \overline{\Tc}. \]
In particular, $\overline{\Sc}$ does not depend on $\mu_0$.
\end{lem}
\begin{proof}
From the hypothesis that the operators $z_n(A_B-z_n I)^{-1}$ are bounded
it follows that the operators $A_B(A_B-z_nI)^{-1} = I + z_n(A_B-z_nI)^{-1}$ are uniformly
bounded. Let $\phi\in H$ be arbitrary. Given $\epsilon>0$ exploit the density of
$D(A_B)$ in $H$ to choose $\psi\in D(A_B)$ such that $\| \phi-\psi \| < \epsilon$. 
Now because $\psi\in D(A_B)$, it follows that $A_B(A_B-z_nI)^{-1}\psi = (A_B-z_n I)^{-1}A_B\psi$
and so
\[ \| A_B(A_B-z_n I)^{-1}\psi \| \leq \| A_B \psi\|\frac{\| z_n(A_B-z_n I)^{-1}\| }{ |z_n| } 
 \rightarrow 0 \;\;\; (n\rightarrow\infty). \]
Hence for all sufficiently large $n$, $\| A_B(A_B-z_n I)^{-1}\psi \| < \epsilon$. But we
know that the operators $A_B(A_B-z_n I)^{-1}$ are uniformly bounded, so for all sufficiently
large $n$
\[ \| A_B(A_B-z_n I)^{-1}\phi \| \leq  \| A_B(A_B-z_n I)^{-1} \| \| \phi - \psi \|
 + \| A_B(A_B-z_n I)^{-1}\psi \| < C\epsilon + \epsilon \]
for some $C>0$. Hence 
\[ \| A_B(A_B-z_n I)^{-1}\phi \| \rightarrow 0 \;\;\; (n\rightarrow \infty) \]
for each fixed $\phi \in H$. Similar arguments may be found in, e.g., \cite[Lemma II.3.4]{kn:engel-nagel}.

Let $\mu_0$ be as in the definition of $\Sc$ and let $\phi = S_{\mu_0,B} f$ for some $f$ in the
boundary space. Evidently $ \| A_B(A_B-z_n I)^{-1}\phi \| \rightarrow 0$ and so
\[ -z_n(A_B-z_n I)^{-1}S_{\mu_0,B} f \rightarrow S_{\mu_0,B} f. \]
It follows from the definition of $\Sc$ that $S_{\mu_0,B} f \in \overline{\Sc}$. Now
if $\mu$ is another point in the resolvent set of $A_B$ then the identity
\[ S_{\mu,B} = S_{\mu_0,B} + (\mu-\mu_0)(A_B-\mu)^{-1} S_{\mu_0,B} \]
from Corollary \ref{corollary:1}  immediately
shows that $S_{\mu,B}f$ lies in $\overline{\Sc}$ also. 
It follows that $\Tc \subseteq  \overline{\Sc}$ and hence $\overline{\Tc} \subseteq \overline{\Sc}$.

Next we show that if $f$ lies in the boundary space and $\mu$, $\delta$ do not lie in $\sigma(A_B)$
then $(A_B-\delta)^{-1}S_{\mu,B} f$ lies in $\overline{\Tc}$.  For this we again use the
formula (\ref{eq:sdiff}) which gives, for $\delta\neq \mu$,
\[ (A_B-\delta)^{-1}S_{\mu,B} f = \frac{1}{\delta-\mu}(S_{\delta,B} f - S_{\mu,B} f); \]
the right hand side of this expression obviously lies in $\Tc$. Taking the
limit as $\mu\rightarrow\delta$ it follows that $(A_B-\delta)^{-1}S_{\delta,B} f$
lies in $\overline{\Tc}$. Thus $\Sc \subseteq \overline{\Tc}$ and $\overline{\Sc}\subseteq
\overline{\Tc}$. 
\end{proof}

\begin{rem} In fact with some mild additional assumptions one may show that $\overline{\Sc}$ is 
generically independent of $B$ (as well as of $\mu_0$), using the identity
\[ S_{\mu_0,C}(I-(C-B)\Gamma_2 S_{\mu_0,B}) = S_{\mu_0,B} \]
from Proposition 4.5 of \cite{kn:BMNW}.
\end{rem}

\begin{rem}The hypothesis that one can choose $(z_n)$ tending to infinity such that $(z_n(A_B-z_nI)^{-1})_{n\in\Nn}$ 
is bounded holds in the case when the numerical range $\omega(A_B)$ is contained in a half plane, for in this case
the $z_n$ can be chosen so that
\[ \frac{z_n}{\mbox{dist}(z_n,\omega(A_B))} \]
is uniformly bounded in $n$.
\end{rem}

\begin{lem}\label{lemma:mm1}
The space $\overline{\Sc}$ is a regular invariant space of the resolvent of the operator $A_B$: that is, 
$\overline{(A_B-\mu I)^{-1}\overline{\Sc}} = \overline{\Sc}$ for all $\mu\in \rho(A_B)$.
\end{lem}
\begin{proof} We start by showing that $(A_B-\mu I)^{-1}\Sc \subseteq \Sc$ for all $\mu\in \rho(A_B)$.
Choose $f$ of the form
\be f = \sum_{j=1}^N  (A_B-\delta_{j}I)^{-1}S_{\mu_0,B}f_{j} 
\label{eq:frep} \ee
for some functions $f_{j}$ in $\Hc$, and note that such $f$ are dense in $\overline{\Sc}$. It 
follows from the resolvent identity 
\be (A_B-\mu I)^{-1}(A_B-\delta I)^{-1} = \frac{1}{\mu-\delta}\{ 
 (A_B-\mu I)^{-1} -  (A_B-\delta I)^{-1}\} \label{eq:resolvent} \ee
that $(A_B-\mu I)^{-1}f$ also admits a representation of the form (\ref{eq:frep}); thus $(A_B-\mu)^{-1}f$ also 
lies in $\Sc$, giving $(A_B-\mu I)^{-1}\Sc \subseteq \Sc$.

Now suppose that $f$ lies in $\overline{\Sc}$. We can write $f=\lim_{N\rightarrow\infty}f_N$
where $f_N$ has the form
\[ f_N = \sum_{j=1}^N (A_B-\delta_{j,N}I)^{-1}S_{\mu_0,B}f_{j,N} \]
and so
\begin{eqnarray*}
 f_N & = & (A_B-\mu I)^{-1}\sum_{j=1}^N (A_B-\mu I) (A_B-\delta_{j,N}I)^{-1}S_{\mu_0,B}f_{j,N} \\
 & & \\
 & = & (A_B-\mu I)^{-1}\sum_{j=1}^N \left\{ S_{\mu_0,B}f_{j,N} 
 + (\delta_{j,N}-\mu)(A_B-\delta_{j,N}I)^{-1}S_{\mu_0,B}f_{j,N}\right\}. 
\end{eqnarray*}
Now the term $\sum_{j=1}^N S_{\mu_0,B}f_{j,N} $ lies in the space $\Tc$ of (\ref{eq:mm1b})
which is contained in $\overline{\Sc}$ by Lemma \ref{lemma:mm0}. Thus $f_N$ has the form
$(A_B-\mu I)^{-1} h_N$ for some $h_N\in \overline{\Sc}$. Hence $f$ lies in
$\overline{(A_B-\mu I)^{-1}\overline{\Sc}}$, in other words 
$\overline{\Sc} \subseteq \overline{(A_B-\mu I)^{-1}\overline{\Sc}}$. 
This completes the proof.
\end{proof}

Corresponding to the spaces $\Sc$ and $\Tc$ we define, from the formally adjoint \footnote{In fact we
showed in \cite{kn:BMNW} that $\At_{B^*}$ is the adjoint of $A_B$.} operators, the spaces
\be \tilde{\Sc} = \mbox{\rm Span}_{\delta\not\in \sigma(\tilde{A}_{B^*})}
 (\tilde{A}_{B^*}-\delta I)^{-1}\mbox{Ran}(\tilde{S}_{\tilde{\mu},B^*})\label{eq:tmm1}, \ee
\be \tilde{\Tc} = \mbox{\rm Span}_{\mu \not\in \sigma(\tilde{A}_{B^*})}
 \mbox{Ran}(\tilde{S}_{\mu,B^*})\label{eq:tmm1b}, \ee
where $\tilde{S}_{\mu,B^*} 
= \left(\left.(\tilde{\Gamma}_1-B^*\tilde{\Gamma}_2) \right|_{\ker(A^*-\mu I)}
\right)^{-1}$ is the corresponding solution operator. Once again, it may be shown that
$\overline{\tilde{\Sc}} = \overline{\Tc}$ and so $\overline{\tilde{\Sc}}$ does not depend
on $\tilde{\mu}$. 


We have so far defined the Weyl function $M_B(\cdot)$ on $\rho(A_B)$ where it is an analytic function. In what follows we will call a point $\lambda_0\in\mathbb{C}$ a point of analyticity of $M_B$ if all analytic continuations of $M_B$ coincide in a neighbourhood of $\lambda_0$.

\begin{thm}\label{theorem:1} 
Suppose that a point $\lambda_0$ is a point of analyticity of $M_B$ and is also a 
limit point of points of analyticity of $\lambda \mapsto (A_B-\lambda I)^{-1}$ -- that is, $\lambda\in
\overline{\rho(A_B)}$. Let $\Sc$ be as in (\ref{eq:mm1}) and, for positive integers $N$ and $M$,
let $P_{N,\Sc}$ and $P_{M,\tilde{\Sc}}$ denote projections onto any $N$ and $M$-dimensional
subspaces of $\Sc$ and $\tilde{\Sc}$ respectively. Then $P_{M,\tilde{\Sc}}(A_B-\lambda I)^{-1}P_{N,\Sc}$ 
is analytic at $\lambda=\lambda_0$. A similar result holds when one uses
projections $P_{N,\Tc}$ and $P_{M,\tilde{\Tc}}$ onto finite-dimensional subspaces of 
$\Tc$ and $\tilde{\Tc}$.
\end{thm}

\begin{proof}

Let $f\in \mbox{Ran}(\Gamma_1-B\Gamma_2)$ and let $F=S_{\mu,B}f$ for $\mu\in \rho(A_B)$. Then for each
$\lambda\in \Cc$,
\be (\At^*-\lambda I)F =(\mu-\lambda)F = (\mu-\lambda)S_{\mu,B}f. \label{eq:atf} \ee
From the resolvent identity (\ref{eq:resolvent}) it follows that for $\lambda,\delta \in \rho(A_B)$,
\begin{eqnarray*} (A_B-\lambda I)^{-1}(A_B-\delta I)^{-1}(\At^*-\lambda I)F \hspace{5cm} \\
 = \frac{\mu-\lambda}{\lambda-\delta}\left\{(A_B-\lambda I)^{-1}S_{\mu,B}f
 - (A_B-\delta I)^{-1}S_{\mu,B}f\right\} 
\end{eqnarray*}
and hence, replacing $(\At^*-\lambda I)F$ on the left hand side by 
$(\mu-\lambda)S_{\mu,B}f$ and the first copy of $S_{\mu,B}f$ on the right
hand side by $(\mu-\lambda)^{-1}(\At^*-\lambda I)F$, 
\begin{eqnarray*}
 (A_B-\lambda I)^{-1}[(A_B-\delta I)^{-1}S_{\mu,B}f] \hspace{5cm} \\
 = \frac{1}{(\mu-\lambda)(\lambda-\delta)}(A_B-\lambda)^{-1}(\At^*-\lambda I)F 
 - \frac{(A_B-\delta I)^{-1}}{\lambda-\delta}S_{\mu,B}f. 
\end{eqnarray*}
Let $v\in D(\A^*)$ and recall that $(\Gamma_1-B\Gamma_2)F=f$. The remainder of our proof relies heavily on the 
identity
\be \left(F-(A_B-\lambda)^{-1}(\At^*-\lambda)F,(A^*-\overline{\lambda}I)v\right)
 = -(f,\tilde{\Gamma}_2v)_{\Hc}
 + (M_B(\lambda)f,(\tilde{\Gamma}_1-B^*\tilde{\Gamma}_2)v)_{\K}
\label{eq:fund} \ee
which is eqn. (5.1) in \cite{kn:BMNW}. Note that on the right hand side of this equation,
the only $\lambda$-dependent term is $M_B(\lambda)$. Using this identity yields
\begin{equation} \label{eq:fund2} 
\begin{array}{c}\hspace{-3.5cm}\left((A_B-\lambda I)^{-1}[(A_B-\delta I)^{-1}S_{\mu,B}f],(A^*-\overline{\lambda}I)v\right) \\ \\
 \hspace{1.2cm} =  \frac{1}{(\mu-\lambda)(\lambda-\delta)} \left\{ \left(F,(A^*-\overline{\lambda}I)v\right)+(f,\tilde{\Gamma}_2v)_{\Hc}
 - (M_B(\lambda)f,(\tilde{\Gamma}_1-B^*\tilde{\Gamma}_2)v)_{\K}\right\} \\ \\
 \hspace{0.5cm}   -\frac{1}{\lambda-\delta}\left((A_B-\delta I)^{-1}S_{\mu,B}f,(A^*-\overline{\lambda}I)v\right) 
\end{array}
\end{equation}
If we now select $N$ points $\delta_j$ in the resolvent set of $A_B$ and $N$ functions $f_j$ in $\mbox{Ran}(\Gamma_1-B\Gamma_2)$,
and define
\[ \Phi := \sum_{j=1}^N(A_B-\delta_j I)^{-1}S_{\mu,B}f_j \in \Sc,\;\;\;  
   \Psi := \sum_{j=1}^N\frac{S_{\mu,B}f_j}{\lambda-\delta_j} \in \Tc, \]
\[ \Theta :=  \sum_{j=1}^N(A_B-\delta_j I)^{-1}   \frac{S_{\mu,B}f_j}{\lambda-\delta_j} \in \Sc, \;\;\;
   \phi :=  \sum_{j=1}^N f_j, \]
then we obtain, upon summing the identities (\ref{eq:fund2}) with $\delta\mapsto\delta_j$ and $f\mapsto f_j$,
\begin{equation} \label{eq:fund3} 
\begin{array}{c}\left((A_B-\lambda I)^{-1}\Phi,(A^*-\overline{\lambda}I)v\right)
 =  -\left(\Theta,(A^*-\overline{\lambda}I)v\right) \\
 \\
 \hspace{0.5cm} + \frac{1}{\mu-\lambda}\left\{ \left( \Psi ,(A^*-\overline{\lambda}I)v\right)
 + (\phi,\tilde{\Gamma}_2v)_{\Hc}
 - (M_B(\lambda)\phi,(\tilde{\Gamma}_1-B^*\tilde{\Gamma}_2)v)_{\K}\right\} 
\end{array}
\end{equation} 
We have thus developed from (\ref{eq:fund}) an expression in which $(A^*-\lambda)F$ has been replaced by 
the arbitrary element $\Phi$ of any finite-dimensional subspace of $\Sc$.  From the right hand side of the expression (\ref{eq:fund3}), since 
$M_B(\lambda)$ is analytic at $\lambda_0$ and since none of the $\delta_j$ is equal to $\lambda_0$,  
it follows that $((A_B-\lambda I)^{-1}\Phi,(A^*-\overline{\lambda}I)v)$ is analytic at $\lambda_0$.
Now the term $(\A^*-\overline{\lambda})v$ may also be turned into an arbitrary element $\tilde{\Phi}$ of 
any finite-dimensional subspace of $\tilde{\Sc}$ by similar reasoning, and so $((A_B-\lambda I)^{-1}\Phi,\tilde{\Phi})$ is analytic
at $\lambda_0$.

The reasoning is similar but slightly simpler when working with elements of $\Tc$.
\end{proof}

In the case of isolated spectral points this theorem can be generalized as follows.

\begin{thm}\label{theorem:2} 
Suppose that a point $\lambda_0$ is a point of analyticity of $M_B$ and that $\lambda_0$ is at worst an 
isolated singularity of $(A_B-\lambda I)^{-1}$ and suppose that the resolvent set $\rho(\tilde{A}_{B^*})$
has finitely many connected components. Let $P_{\overline{\Sc}}$ and $P_{\overline{\tilde{\Sc}}}$
denote orthogonal projections onto the closures of $\Sc$ and $\tilde{\Sc}$ respectively.
Then $P_{\overline{\tilde{\Sc}}}(A_B-\lambda I)^{-1}P_{\overline{\Sc}}$ 
is analytic at $\lambda=\lambda_0$.
\end{thm}

\begin{proof} Assume that $\lambda\not\in \rho(A_B)$ otherwise the statement is trivial.
In eqn. (\ref{eq:fund3}) take $v = (\tilde{S}_{\tilde{\mu},B^*})g$ for any $g\in \mbox{Ran}(\Gammat_1-B^*\Gammat_2)$ and any $\tilde{\mu}$ not in
the spectrum of $\tilde{A}_B^*$. Then $(A^*-\overline{\lambda} I)v = (\tilde{\mu}-\overline{\lambda})(\tilde{S}_{\tilde{\mu},B^*})g$ and 
so from (\ref{eq:fund3}),
\begin{equation} \label{eq:fund4} 
\begin{array}{c}\left((A_B-\lambda I)^{-1}\Phi,(\tilde{S}_{\tilde{\mu},B^*})g\right)
 =  -\frac{1}{\overline{\tilde{\mu}}-\lambda}\left(\Theta,(A^*-\overline{\lambda})v\right) \\
 \\
 \hspace{0.5cm} + \frac{1}{(\overline{\tilde{\mu}}-\lambda)(\mu-\lambda)}\left\{ \left( \Psi ,(A^*-
  \overline{\lambda})v\right)
 + (\phi,\tilde{\Gamma}_2v)_{\Hc}
 - (M_B(\lambda)\phi,(\tilde{\Gamma}_1-B^*\tilde{\Gamma}_2)v)_{\K}\right\} 
\end{array}
\end{equation} 
Since $\tilde{\mu}$ lies in the resolvent set of $\tilde{A}_{B^*} = (A_B)^*$ we know that $\tilde{\mu}\neq \overline{\lambda_0}$. Let $\Gamma$ be any smooth closed contour surrounding $\lambda_0$, not enclosing $\mu$ or 
$\overline{\tilde{\mu}}$ and bounded away from the spectrum of $A_B$. Integrating (\ref{eq:fund4}) around
$\Gamma$ yields
\[ \int_\Gamma \left( (A_B-\lambda I)^{-1}\Phi, (\tilde{S}_{\tilde{\mu},B^*})g\right) d\lambda = 0. \]
It follows that for any $\hat{\Phi}$ having a representation of the form
\be \hat{\Phi} = \sum_{j=1}^M (\tilde{S}_{\tilde{\mu}_j,B^*})g_j \label{eq:tpf} \ee
in which the points $\overline{\tilde{\mu}_j}$ lie outside $\Gamma$, we have
\be \int_\Gamma \left( (A_B-\lambda I)^{-1}\Phi, \hat{\Phi} \right) d\lambda = 0. \label{eq:ephat} \ee
Consider now a general $\tilde{\Phi}$ in $\overline{\Sc} = \overline{\Tc}$. Given $\epsilon>0$, such a $\tilde{\Phi}$
can be approximated to accuracy $\epsilon$ by $\tilde{\Phi}_\epsilon$ of the form 
\be \tilde{\Phi}_\epsilon = \sum_{j=1}^M (\tilde{S}_{\tilde{\mu}_{j,\epsilon},B^*})g_{j,\epsilon} \label{eq:tpfe} \ee
in which the points $\overline{\tilde{\mu}_{j,\epsilon}}$ could, however, lie inside $\Gamma$. However the solution operator 
$\tilde{S}_{\tilde{\mu},B^*}$ is analytic for $\tilde{\mu}$ in the resolvent set $\rho(\tilde{A}_{B^*})$. If 
the curve $\Gamma$ is chosen in a sufficiently small neighbourhood of $\lambda_0$ then its image under complex 
conjugation, denoted $\overline{\Gamma}$, lies in a single connected component of the resolvent set 
$\rho(\tilde{A}_{B^*})$. Denote this connected component by ${\mathcal U}$ and choose any open set ${\mathcal O}$ in 
${\mathcal U}$ outside $\overline{\Gamma}$. The values of the analytic function 
$\tilde{\mu}\mapsto \tilde{S}_{\tilde{\mu},B^*}$ at any point in ${\mathcal U}$ (and hence, in particular, at
any points $\tilde{\mu}_{j,\epsilon}$ inside $\overline{\Gamma}$) are uniquely determined by the values of this function in 
${\mathcal O}$, so it must be possible to approximate $\tilde{\Phi}_\epsilon$ of the form (\ref{eq:tpf}) to accuracy 
$\epsilon$ by approximations of the form
\be \hat{\Phi}_\epsilon = \sum_{j=1}^{K} (\tilde{S}_{\zeta_{j,\epsilon},B^*})h_{j,\epsilon} \label{eq:tpf2} \ee
in which the points $\zeta_{j,\epsilon}$ either lie in ${\mathcal O}$ or in a completely different component
of the resolvent set $\rho(\tilde{A}_{B^*})$. We have $\| \tilde{\Phi}-\hat{\Phi}_\epsilon \| < 2\epsilon$ and 
we also have, from (\ref{eq:ephat}), 
\be \int_\Gamma \left( (A_B-\lambda I)^{-1}\Phi, \hat{\Phi}_\epsilon \right) d\lambda = 0. \label{eq:ephat2} \ee
Since the vectors $(A_B-\lambda I)^{-1}\Phi$ are uniformly bounded on $\Gamma$, which does not intersect
the spectrum of $A_B$, we can take limits in $\epsilon$ and obtain
\be \int_\Gamma \left( (A_B-\lambda I)^{-1}\Phi, \tilde{\Phi} \right) d\lambda = 0 \label{eq:ephat3} \ee
for all $\Phi\in \overline{\Sc}$, $\tilde{\Phi}\in \overline{\tilde{\Sc}}$. The result is now immediate
from Morera's theorem.
\end{proof}

\section{A first-order example}\label{section:4}
\label{section:1storder}
An obvious question arising from the previous section is whether or not the result 
of Theorem \ref{theorem:1} remains true if one omits projections onto finite dimensional subspaces: if $M_B(\lambda)$ 
is analytic at some point which is a non-isolated spectral point of $A_B$, is 
$P_{\overline{\tilde{S}}}(A_B-\lambda I)^{-1}P_{\overline{S}}$ also analytic at this point? A simple example
shows that this result is false.

Consider in $L^2(0,\infty)$ the operator $A = \tilde{A}$ given by $D(A)=H^1_0(0,\infty)$ with
\be Af = i\frac{df}{dx}. \label{eq:Adef} \ee
The operator $A$ is maximal symmetric and $D(A^*)=H^1(0,\infty)$.
Define the boundary spaces $\cH = \Cc$, $\cK = \{ 0 \}$, and 
boundary value operators $\Gamma_1$, $\Gamma_2$, $\Gammat_1$, $\Gammat_2$ by 
\be \Gamma_1 f = if(0), \;\;\; \Gammat_2 f = f(0). \label{eq:Gammadef} \ee
\be \Gammat_1 f = 0, \;\;\; \Gamma_2 f = 0. \label{eq:Gammadeft} \ee
It is easy to see that the pairs $(\Gamma_1,\Gamma_2)$ and $(\Gammat_1,\Gammat_2)$ are surjective
and a simple integration shows that 
\[ \left(A^*f,g\right) - \left(f,A^*g\right) = i f(0)\overline{g(0)} = \Gamma_1 f \overline{\Gammat_2g} -   \Gamma_2 f \overline{\Gammat_1g}. \]
Because $\Gammat_1$ and $\Gamma_2$ are trivial it follows immediately from the definitions that
\[ M_B(\lambda) = 0; \;\;\; \tilde{M}_{\tilde{B}}(\lambda) = -1/\tilde{B}. \]
Moreover, $\sigma(A_B)=\overline{\mathbb{C}^+}$.

Now we consider the space $\overline{\Tc}$, for simplicity in the case $B=0$. For $\Im(\mu)<0$ a typical element of 
$\Tc$ has the form $y_\mu = S_{\mu,0}f$ and therefore satisfies $iy_\mu' = \mu y_\mu$ with $y_\mu(0)=f$; in
other words, for some complex number $f$,
\[ y_\mu(x) = f\exp(-i\mu x). \]
Now suppose that $u\in \overline{\Tc}^\perp$. Then $(u,y_\mu)=0$ for all $\Im(\mu)<0$. This means
\[ \int_0^\infty u(x) \exp(i\mu x)dx = 0 \]
for all $\Im(\mu)<0$. Setting $\mu = \omega - ir$, $r>0$, we deduce that for all $\omega\in \Rr$
\[ \int_{0}^\infty u(x)\exp(-rx)\exp(i\omega x)dx = 0. \]
From inverse Fourier transformation this implies that \mbox{$u(x)\exp(-rx) = 0$} for all
$x$ and hence $u(x)\equiv 0$. Thus we have proved that for this example,
\[ \overline{\Tc} = \overline{\tilde{\Tc}} = L^2(0,\infty) \]
and so $(A_B-\lambda I)^{-1}$ is not reduced by the bordering projection operators $P_{\overline{\Tc}}$ and 
$P_{\overline{\tilde{\Tc}}}$. It follows that for this example, the set of singular points of the bordered
resolvent is strictly greater than the set of singular points of $M_B(\lambda)$.

\section{A Hain-L\"{u}st type example}\label{section:5}
In this section we consider a block operator matrix example in which $P_{\overline{\tilde{\Sc}}}(A_B-\lambda I)^{-1}P_{\overline{S}}$ 
has exactly the same singularities as $M_B(\lambda)$, even though some of these singularities are not isolated. In other
words, for the example which we present here, a stronger result holds than those available in Theorems \ref{theorem:1}
and \ref{theorem:2}. It is not yet clear to us what special properties of this example mean that, unlike for the example
of Section \ref{section:1storder}, better results hold here than those in Theorems \ref{theorem:1} and \ref{theorem:2}.

Let 
\begin{equation}
\At^* = \left(\begin{array}{cc} -\frac{d^2}{dx^2}+q(x) & w(x) \vspace{2pt}\\
 w(x) & u(x) \end{array}\right), 
\label{eq:hl1} 
\end{equation}
where $q$, $u$ and $w$ are complex-valued $L^\infty$-functions, and the domain of
the operator is given by
\begin{equation}
D(\At^*) = H^2(0,1)\times L^2(0,1). 
\label{eq:hl2} 
\end{equation}
Also let 
\begin{equation}
 A^* = \left(\begin{array}{cc} -\frac{d^2}{dx^2}+\overline{q(x)} & \overline{w(x)} \vspace{2pt}\\ 
 \overline{w(x)} & \overline{u(x)} \end{array}\right), \;\;\;
 \mbox{with $D(A^*)=D(\At^*)$}. 
\label{eq:hl3}
\end{equation}
It is then easy to see that
\begin{eqnarray}
\llangle \At^*\left(\begin{array}{c} y \\ z \end{array}\right),
 \left(\begin{array}{c} f \\ g \end{array}\right)\rrangle
 -  \llangle \left(\begin{array}{c} y \\ z \end{array}\right),
  A^*\left(\begin{array}{c} f \\ g \end{array}\right)\rrangle\nonumber & & \\
 & \hspace{-10cm} = & \hspace{-5cm}\llangle \Gamma_1\left(\begin{array}{c} y \\ z \end{array}\right),
 \Gamma_2\left(\begin{array}{c} f \\ g \end{array}\right)\rrangle
 - \llangle \Gamma_2\left(\begin{array}{c} y \\ z \end{array}\right),
 \Gamma_1\left(\begin{array}{c} f \\ g \end{array}\right)\rrangle, 
\label{eq:hl4}
\end{eqnarray}
where 
\[ \Gamma_1\left(\begin{array}{c} y \\ z \end{array}\right)
  = \left(\begin{array}{c} -y'(1) \\ y'(0) \end{array}\right),
\;\;\;
 \Gamma_2\left(\begin{array}{c} y \\ z \end{array}\right)
  = \left(\begin{array}{c} y(1) \\ y(0) \end{array}\right).
\]
Consider the operator
\begin{equation}
 A_{\alpha\beta} := \left. \At^*\right|_{\mbox{ker}(\Gamma_1-B\Gamma_2)},
\label{eq:hl5}
\end{equation}
where, for simplicity, 
\be B = \left(\begin{array}{cc} \cot\beta & 0 \\ 
                0 & -\cot\alpha \end{array}\right). \label{eq:Bdef} \ee
It is known \cite{ALMS} that 
\[ \sigma_{ess}(A_{\alpha\beta}) = \mbox{essran}(u) := \left\{ z \in \Cc \, | \forall \epsilon > 0, \;
 \mbox{\rm meas} \left(\{ x \in [0,1] \, | \, |u(x)-z|<\epsilon \} \right) > 0 \right\} . \]
This result is independent of the choice of boundary conditions. The measure used is Lebesgue.
Note also that $\sigma(A_{\alpha\beta})$ is not the whole of $\Cc$ for essentially bounded $q$, $u$ and $w$.
For future use we also define the set
\[ \Wc = \{ x \in [0,1] \, | \, w(x) \neq 0\}. \]
The function $w$ is defined only almost everywhere, but this is sufficient to
define $\Wc$ up to a set of measure zero, which can be neglected. 
 
We now calculate the function $M(\lambda) = \left(\begin{array}{cc} m_{11}(\lambda) & m_{12}(\lambda) \\
 m_{21}(\lambda) & m_{22}(\lambda) \end{array}\right)$ such that
\[ M(\lambda)(\Gamma_1 - B \Gamma_2)\left(\begin{array}{c} y \\ z 
 \end{array}\right) = \Gamma_2\left(\begin{array}{c} y \\ z 
 \end{array}\right) \]
for $ \left(\begin{array}{c} y \\ z  \end{array}\right)\in
 \mbox{ker}(\At^*-\lambda)$. In our calculation we assume
that $\lambda\not\in \sigma_{ess}(A_{\alpha\beta})$. The condition
$ \left(\begin{array}{c} y \\ z  \end{array}\right)\in
 \mbox{ker}(\At^*-\lambda)$ yields the equations
\[ -y'' + (q-\lambda)y + wz = 0;\;\;\
 wy + (u-\lambda)z = 0 \]
which, in particular, give
\begin{equation}
 -y'' + (q-\lambda)y + \frac{w^2}{\lambda-u}y = 0.
\label{eq:hl10}
\end{equation}
The linear space $\mbox{ker}(\At^*-\lambda)$ is thus
spanned by the functions 
$ \left(\begin{array}{c} y_1 \\ wy_1/(\lambda-u)  \end{array}\right)$
and 
$ \left(\begin{array}{c} y_2 \\ wy_2/(\lambda-u)  \end{array}\right)$
where $y_1$ and $y_2$ are solutions of the initial value problems
consisting of the differential equation (\ref{eq:hl10}) equipped
with initial conditions
\begin{equation}
 y_1(0) = \cos\alpha, \;\; y_1'(0) = \sin\alpha, \label{eq:hl10a}
\end{equation}
\begin{equation}
 y_2(0) = -\sin\alpha, \;\; y_2'(0) = \cos\alpha, \label{eq:hl10b}
\end{equation}
where $\alpha$ is as in (\ref{eq:Bdef}). A straightforward calculation shows that
\[ \left(\begin{array}{c} y(1) \\ y(0) \end{array}\right)
 = \left(\begin{array}{cc} m_{11}(\lambda) & m_{12}(\lambda) \\
 m_{21}(\lambda) & m_{22}(\lambda) \end{array}\right)
\left(\begin{array}{c} -y'(1) -\cos\beta\ y(1)/\sin\beta\\ y'(0)
 + \cos\alpha\ y(0)/\sin\alpha \end{array}\right). \]
Note that the $y_j$ depend on $x$ and $\lambda$ but that the
$\lambda$-dependence is suppressed in the notation, except when necessary. 
Another elementary calculation now shows that
\begin{equation}
 m_{11}(\lambda) = -\frac{y_2(1,\lambda)}{y_2'(1,\lambda) + 
 \cot\beta\ y_2(1,\lambda)}, 
\label{eq:hl7} 
\end{equation}
\begin{equation}
 m_{21}(\lambda) = m_{12}(\lambda) = \frac{\sin\alpha}{y_2'(1,\lambda) + 
 \cot\beta\ y_2(1,\lambda)}, 
\label{eq:hl8} 
\end{equation}
\begin{equation}
 m_{22}(\lambda) = \sin\alpha\cos\alpha+\sin^2\alpha\left\{
\frac{y_1'(1,\lambda) + \cot\beta\ y_1(1,\lambda)}{y_2'(1,\lambda) + 
 \cot\beta\ y_2(1,\lambda)}\right\}.
\label{eq:hl9} 
\end{equation}

As an aside, notice that all these expressions contain a denominator 
$y_2'(1,\lambda) +  \cot\beta\ y_2(1,\lambda)$ and that 
$\lambda\not\in \mbox{essran}(u|_{\Wc})$ is an eigenvalue precisely when
this denominator is zero. 

\begin{rem}\label{rem:poles}
For $\lambda\in\Cc\setminus \mbox{essran}(u|_{\Wc})$, 
the coefficient $w(x)^2/(u(x)-\lambda)$ in (\ref{eq:hl10}) is analytic as a function of $\lambda$. Therefore, the solutions $y_1$ and $y_2$ are analytic in $\lambda$.
The $M$-function may have an isolated 
pole at some point $\lambda$ if 
$y_2'(1,\lambda) +  \cot\beta\ y_2(1,\lambda)$
happens to be zero; such a pole will be an eigenvalue of the operator
$A_{\alpha\beta}$ and may or may not be embedded in the essential spectrum of the operator.
\end{rem}
As a consequence of this remark, the $M$-function can be analytic at points in the essential range of $u$, as long as those points are outside the essential range of $ u\vert_{\Wc}$:
\begin{lem} 
Apart from poles at eigenvalues of $A_{\alpha\beta}$, the $M$-function 
$M(\lambda)$ is analytic in the set $\Cc\setminus\essran\left(\left. u\right|_{\Wc}\right)$.
\end{lem}
We now turn our attention to the behaviour of the resolvent $(A_{\alpha\beta}-\lambda I)^{-1}$
on the spaces $\Tc$ and $\overline{\Tc}$. 

\begin{thm}\label{theorem:4}
\be \overline{\Sc}=\overline{\Tc} \subseteq \left(\begin{array}{c} L^2(0,1) \\
                                                          L^2(\Wc) \end{array}\right). \label{eq:hltc}\ee
Moreover if $M_B(\lambda)$ is analytic at a point $\lambda$ not in
$\essran\left(\left. u\right|_{\Wc}\right)$ then 
\be \left(\begin{array}{c} y \\ z \end{array}\right) := (A_{\alpha\beta}-\lambda I)^{-1}
\left(\begin{array}{c} f \\ g \end{array}\right) \label{eq:ARS} \ee
admits analytic continuation for any $f\in L^2(0,1)$ and $g\in L^2(\Wc)$. 
\end{thm}

\begin{proof}
Suppose that $(f_1,f_2)\in\Cc^2$ and that $\mu \in \rho(A_{\alpha\beta})$. Since $\mu$ does not lie in
the essential spectrum, it does not lie in the essential range of $u$, so $1/(u-\mu)$ is essentially
bounded. Consider the functions $y_\mu$, $z_\mu$ defined by
\[ \left(\begin{array}{c} y_\mu \\ z_\mu \end{array}\right) = S_{\mu,B}\left(\begin{array}{c} f_1 \\ f_2 \end{array}\right) ; \]
eliminating $z_\mu$ from these equations using
\be z_\mu = \frac{ w y_\mu}{u-\mu} \label{eq:zmu} \ee
we find that $y_\mu$ satisfies the ODE 
\be -y_\mu'' + (q-\mu)y_\mu + \frac{w^2}{\mu-u}y_\mu = 0 \label{eq:hmu} \ee
with boundary conditions $y_\mu'(1)+\cot(\beta)y_\mu(1) = -f_1$ and $y_\mu'(0) + \cot(\alpha)y_\mu(0)=f_2$.  The boundary
value problem for $y_\mu$ is uniquely solvable because $\mu\in \rho(A_{\alpha\beta})$ and so 
$y_{\mu}\in L^2(0,1)$.  It follows from (\ref{eq:zmu}) that $z_{\mu}\in L^2(\Wc)$.
This proves the inclusion (\ref{eq:hltc}).

We decompose the space 
\be
\left(\begin{array}{cc}L^2(0,1) \\ L^2(0,1)\end{array}\right)=\left(\begin{array}{cc}L^2(0,1) \\ L^2(\Wc)\end{array}\right)\bigoplus\left(\begin{array}{cc} 0 \\ L^2(\Wc^c)\end{array}\right)
\ee
where $\Wc^c=[0,1]\setminus \Wc$. Denote $H_1=\left(\begin{array}{cc}L^2(0,1) \\ L^2(\Wc)\end{array}\right)$ and $H_2=\left(\begin{array}{cc} 0 \\ L^2(\Wc^c)\end{array}\right)$.
We shall now show that these are reducing subspaces for the operator $A_{\alpha\beta}$.
It is clear that if $\left(\begin{array}{cc}h \\ g\end{array}\right)\in D(A_{\alpha\beta})$ then the projections of $\left(\begin{array}{cc}h \\ g\end{array}\right)$ onto $H_1$ and $H_2$ will also lie in the domain of the operator as $H_2\subseteq D(A_{\alpha\beta})$. The conditions $A_{\alpha\beta}P_{H_i}\left(\begin{array}{cc}h \\ g\end{array}\right)\in H_i$ when $\left(\begin{array}{cc}h \\ g\end{array}\right)\in D(A_{\alpha\beta})$ for $i=1,2$ are a simple calculation. Here $P_i$ denotes the orthogonal projection onto $H_i$.

We have $\sigma_{ess}(A_{\alpha\beta}|_{H_1}) = \mbox{essran}(u|_\Wc)$. By Remark \ref{rem:poles}, any eigenvalue of the operator $A_{\alpha\beta}|_{H_1}$ will be a pole of $M_B(\lambda)$. 
Hence, if $M_B(\lambda)$ is analytic at a point $\lambda$ not in
$\essran\left(\left. u\right|_{\Wc}\right)$, we have that $\lambda\in\rho(A_{\alpha\beta}|_{H_1})$ and for 
any $\left(\begin{array}{cc}f \\ g\end{array}\right)\in H_1$, $ (A_{\alpha\beta}-\lambda I)^{-1}
\left(\begin{array}{c} f \\ g \end{array}\right)  $ admits analytic continuation.
\end{proof}

As an immediate corollary of this theorem we have
\begin{cor} 
For  $\lambda\not\in \mbox{\rm essran}\left(\left. u\right|_{\Wc}\right)$ the bordered resolvent
$P_{\overline{\tilde{\Sc}}}(A_{\alpha\beta}-\lambda I)^{-1}P_{\overline{\Sc}}$
is analytic precisely where $M_B(\lambda)$ is analytic.
\end{cor}
\begin{proof}
Since $(A_{\alpha\beta}-\lambda I)^{-1}|_{H_1}$ is analytic on the space $H_1$ which is 
larger than $\overline{\Sc}$ by Theorem \ref{theorem:4}, it is immediate that the bordered
resolvent is analytic wherever $M_B(\cdot)$ is analytic. The fact that $M_B(\cdot)$ is analytic
whenever the bordered resolvent is analytic follows from (\ref{eq:fund3}).
\end{proof}

\begin{rem}
Generically one expects that $M_B(\cdot)$ will not be analytic at points in 
$\mbox{\rm essran}\left(\left. u\right|_{\Wc}\right)$. The analyticity or otherwise depends
on the analyticity or otherwise of solutions of the ODE (\ref{eq:hl10}).
\end{rem}
It is worth mentioning also the following result.
\begin{prop}\label{theorem:5a} 
Let $\lambda$ be any fixed point in the resolvent set $\rho(A_{\alpha\beta})$. Then 
\be \overline{(A_{\alpha\beta}-\lambda I)^{-1}\left(\begin{array}{c} L^2(0,1) \\
                                                          L^2(\Wc) \end{array}\right)}
 = \left(\begin{array}{c} L^2(0,1) \\
                                                          L^2(\Wc) \end{array}\right). 
 \label{eq:reduce} \ee
\end{prop}

\begin{proof}
The domain of the differential expression
\[ -\frac{d^2}{dx^2}+q-\lambda - \frac{w^2}{u-\lambda} \]
equipped with boundary conditions $y(1)+\cot(\beta)y(1) = 0$ and $y'(0) + \cot(\alpha)y(0)=0$, is dense in
$L^2(0,1)$. Thus any function in $L^2(0,1)$ can be approximated to arbitrary accuracy by a solution $y$ of
a boundary value problem
\[ (-\frac{d^2}{dx^2}+q-\lambda - \frac{w^2}{u-\lambda})y = h\in L^2(0,1) \;\;\; y(1)+\cot(\beta)y(1) = 0 = y'(0) + \cot(\alpha)y(0) \]
for a suitably chosen $h$. Having fixed such $y$ and $h$, then for any $z\in L^2(\Wc)$ we may define
$g$ to satisfy
\[ z = \frac{1}{u-\lambda}(g-wy) \]
and clearly have $g\in L^2(\Wc)$. Finally we define $f\in L^2(0,1)$ by
$f = h  + wg/(u-\lambda)$ so that $f\in L^2(0,1)$ and $h = f-wg/(u-\lambda)$. We thus have
\be \label{eq:zy} (-\frac{d^2}{dx^2}+q-\lambda - \frac{w^2}{u-\lambda})y = f-wg/(u-\lambda), \;\;\;
 z = \frac{1}{u-\lambda}(g-wy). 
 \ee
This is equivalent to (\ref{eq:ARS}). We have therefore approximated
an arbitrary element of $\left(\begin{array}{c} L^2(0,1) \\
                                                          L^2(\Wc) \end{array}\right)$
by a function in $(A_{\alpha\beta}-\lambda I)^{-1}\left(\begin{array}{c} L^2(0,1) \\
                                                          L^2(\Wc) \end{array}\right)$.
To get the opposite inclusion consider 
\[ \left(\begin{array}{c} y \\ z \end{array}\right) = (A_{\alpha\beta}-\lambda I)^{-1}\left(\begin{array}{c} f \\
                                                          g \end{array}\right) \]
in which $g \in L^2(\Wc)$. We need to show that $z \in L^2(\Wc)$ also.
The expression for $z$ is given in (\ref{eq:zy}); evidently $wy\in L^2(\Wc)$ and
$g \in L^2(\Wc)$ so the result is immediate.
\end{proof}

\section{A perturbed multiplication operator in $L^2(\Rr)$}\label{section:6}
The results of the foregoing sections show that there are often wide gaps between what may be true
at an abstract level about the relationship between resolvents and $M$-functions, and what may be
achievable in concrete examples.

In light of these gaps, in this section we consider boundary triplets and Weyl $M$-functions
for a simple Friedrichs model with a singular perturbation. Our purpose is to show even more unexpected and counter-intuitive results. For example, in  \cite[Section 4]{kn:BMNW} it  is shown that isolated eigenvalues  of an operator correspond to  isolated poles of the associated $M$-function assuming unique continuation holds, i.e.
\[\ker(\At^*-\lambda)\cap\ker(\Gamma_1)\cap\ker(\Gamma_2)=\{0\},\]
while \cite[Proposition 5.2]{MM02} shows this result under the assumption that the point under consideration is in the resolvent set of an extension of the minimal operator.
In this section, we shall find that these hypotheses which have seemed 
reasonable in the development of an abstract theory of boundary triplets 
are not satisfied by a rather simple example. As a consequence, the 
relationship between the $M$-function and the spectrum of the operator 
becomes more interesting.

We consider in $L^2(\Rr)$ the operator $A$ with domain given by
\be D(A) = \left\{ f\in L^2(\Rr) \, | \, xf(x)\in L^2(\Rr), \;\;\;
 \lim_{R\rightarrow\infty}\int_{-R}^{R}f(x)dx \;\; \mbox{exists and is zero}\right\}, \label{eq:1} \ee
given by the expression
\be (Af)(x) = x f(x) + \langle f,\phi\rangle \psi(x), \label{eq:2} \ee
where $\phi$, $\psi$ are in $L^2(\Rr)$. Observe that since the constant
function ${\mathbf 1}$ does not lie in $L^2(\Rr)$ the domain of $A$ is 
dense in $L^2(\Rr)$. 

Formally, the expression  $xf(x) + \langle f,\phi\rangle\psi(x)$ is
equivalent, by Fourier transformation, to a sum of a first order
differential operator and an inner product (integral) term acting on
the Fourier transform $\hat{f}$. The condition $\int_{\Rr}f = 0$ is 
equivalent to a `boundary' condition $\hat{f}(0) = 0$.

\begin{lem}\label{lemma:1}
The adjoint of $A$ is given on the domain
\be D(A^*) = \left\{ f\in L^2(\Rr) \, | \, \exists c_f\in \Cc : xf(x)-c_f{\mathbf 1} 
\in L^2(\Rr)\right\}, \label{eq:3} \ee
by the formula
\be A^*f = x f(x) - c_f{\mathbf 1} +\langle f,\psi\rangle\phi. \label{eq:4}\ee
\end{lem}
\begin{proof}
Suppose that $f\mapsto \langle A f , g\rangle$ is a bounded linear
functional on $D(A)$. A direct calculation shows that
\[ \langle A f , g\rangle = \int_{\Rr} f(x)\overline{\left(x g(x) + \langle g,\psi\rangle \phi(x)\right)}dx. \]
(Note that the integral is convergent since $xf(x)\in L^2(\Rr)$ and $g\in L^2(\Rr)$.)
In view of the constraint $\int_{\Rr}f = 0$ and the density of $D(A)$ in
$L^2(\Rr)$, the $L^2(\Rr)$-boundedness of this functional implies that for 
some constant $c_g$,
\[ xg+\langle g,\psi\rangle \phi = c_g{\mathbf 1} + h \]
for some $h\in L^2(\Rr)$. Since $\phi\in L^2(\Rr)$ this implies that
$xg- c_g{\mathbf 1}\in L^2(\Rr)$ and so
\[ \langle A f , g\rangle = 
\langle f, xg- c_g{\mathbf 1}+\langle g,\psi\rangle \phi\rangle. \]
The density of $D(A)$ in $L^2(\Rr)$ now gives
$A^*g = xg- c_g{\mathbf 1}+\langle g,\psi\rangle \phi$.
\end{proof}
\begin{rem}
For $f$ sufficiently well behaved at infinity, the constant $c_f$ appearing in Lemma
\ref{lemma:1} is given by
\[ c_f = \lim_{x\rightarrow\infty} xf(x). \]
\end{rem}

For later reference we can calculate the deficiency indices of $A$. To this
end we may neglect the finite rank term $\langle \cdot,\phi\rangle\psi$ and 
calculate the set of $u$ such that
\[ x u(x) - c_u {\mathbf 1} = \pm i u. \]
This yields $u = c_u\frac{\mathbf 1}{x\mp i}$; the factor $c_u$ is a normalization.
A simple calculation shows that $c_{(x\mp i)^{-1}}=1$ and so we may choose
\[ u(x) = (x\mp i)^{-1} \]
as the deficiency elements, showing that $A$ has deficiency indices 
$(1,1)$.

We now introduce `boundary value' operators $\Gamma_1$ and $\Gamma_2$ on 
$D(A^*)$ as follows:
\be \Gamma_1 u = \int_{\Rr} (u(x) - c_u{\mathbf 1}\sign(x)(x^2+1)^{-1/2})dx,\;\;\;
 \Gamma_2 u = c_u. \label{eq:5}\ee
We make the following observations.

\begin{lem}\label{lemma:2}
The operators $\Gamma_1$ and $\Gamma_2$ are bounded relative to $A^*$. 
\end{lem}
\begin{proof}
Observe that
\[ c_u = -A^*u + x u + \langle u,\psi\rangle\phi. \]
Multiply both sides by the characteristic function $\chi_{(0,1)}$
of the interval $(0,1)$, then take $L^2$-norms, to obtain
\[ |c_u| \leq \| A^*u \| + \| u \| + \| u \| \| \psi \| \| \phi \| \]
which shows that $\Gamma_2$ is bounded relative to $A^*$. Similarly, 
an elementary calculation shows that
\[ \Gamma_1 u = \int_{\Rr} \{ (\sqrt{x^2+1}\mbox{sign}(x)-x)u(x) + 
(x u(x) - c_u{\mathbf 1})\}\frac{\mbox{sign}(x)}{\sqrt{x^2+1}}dx; \]
since $(\sqrt{x^2+1})\mbox{sign}(x)-x\in L^2(\Rr)$, this shows that $\Gamma_1$
is bounded relative to $A^*$.
\end{proof}

\begin{lem}\label{lemma:3}
The following `Green's identity' holds:
\be \langle A^*f,g\rangle - \langle f,A^*g \rangle
 = \Gamma_1 f \overline{\Gamma_2 g} - \Gamma_2 f \overline{\Gamma_1 g}
 + \langle f,\psi\rangle \langle\phi,g\rangle
 - \langle f,\phi\rangle \langle \psi,g\rangle. \label{eq:6}
\ee
Consequently, in the case when $\phi=\psi$, the operators 
$\left. A^*\right|_{\ker(\Gamma_1-B\Gamma_2)}$ are selfadjoint for 
any real number $B$.
\end{lem}
\begin{proof}
The identity (\ref{eq:6}) is a simple calculation. In the case
when $\phi = \psi$ the operator $A$ is symmetric and the
selfadjointess of the extensions $\left. A^*\right|_{\ker(\Gamma_1-B\Gamma_2)}$
is a well known result from theory of boundary value spaces: see,
e.g., Gorbachuk and Gorbachuk \cite{kn:GG}.
\end{proof}

In the case when $\phi\neq \psi$, the terms 
$\langle f,\psi\rangle \langle\phi,g\rangle 
- \langle f,\phi\rangle \langle \psi,g\rangle$
on the right hand side of (\ref{eq:6}) arise from the fact that
$A^*$ is not an extension of $A$. In order to eliminate these
terms we follow the formalism of Lyantze and Storozh \cite{Lyantze} 
and introduce an operator $\tilde{A}$ in which $\phi$ and
$\psi$ are swapped:
\be D(\At) = \left\{ f\in L^2(\Rr) \, | \, xf(x)\in L^2(\Rr), \;\;\;
 \lim_{R\rightarrow\infty}\int_{-R}^{R}f(x)dx = 0\right\}, \label{eq:1t} \ee
\be (\At f)(x) = x f(x) + \langle f,\psi\rangle \phi. \label{eq:2t} \ee
In view of Lemma \ref{lemma:1} we immediately see that
$D(\At^*) = D(A^*)$ and that
\be \At^*f = x f(x) - c_f{\mathbf 1} +\langle f,\phi\rangle\psi. \label{eq:4t}\ee
Thus $\At^*$ is an extension of $A$,  $A^*$ is an extension of $\At$, 
and the following result is easily proved.

\begin{lem}\label{lemma:4}
\be A = \left. \At^*\right|_{\ker(\Gamma_1)\cap\ker(\Gamma_2)}; \;\;\;
 \At = \left. A^*\right|_{\ker(\Gamma_1)\cap\ker(\Gamma_2)}; \label{eq:5t}
\ee
moreover, the Green's formula (\ref{eq:6}) can be modified to
\be \langle A^*f,g\rangle - \langle f,\At^*g \rangle
 = \Gamma_1 f \overline{\Gamma_2 g} - \Gamma_2 f \overline{\Gamma_1 g}.
 \label{eq:6t}
\ee
\end{lem}
This is a slight simplification of the situation in \cite{Lyantze}
as only two boundary operators are required, rather than four.

For any fixed complex number $B$ and suitable $\lambda\in \Cc$, by the
`Weyl function $M_B(\lambda)$' we shall
mean the map
\be M_B(\lambda) := \Gamma_2 \left(\left.(\Gamma_1-B\Gamma_2)\right|_{\ker(\At^*-\lambda)}\right)^{-1}. 
\label{eq:mdef}
\ee

We now calculate $M_B(\lambda)$. Suppose that $\Im\lambda\neq 0$ and 
that $f\in \mbox{ker}(\At^*-\lambda I)$.  Then
\[ x f(x) - c_f + \langle f,\phi\rangle \psi = \lambda f \]
and simple algebra yields
\be f = \frac{c_f - \langle f,\phi\rangle \psi}{x-\lambda}. \label{eq:extra2}
\ee
Taking inner products with $\phi$ and recalling that $\Gamma_2f = c_f$ yields
\be \langle f, \phi\rangle D(\lambda)
 = \Gamma_2 f \langle (x-\lambda)^{-1},\phi\rangle \label{eq:extra1} \ee
where $D(\lambda) = 1 + \int_{\Rr}(x-\lambda)^{-1}\psi\overline{\phi}dx$.
Substituting back into (\ref{eq:extra2}) yields
\be f = \Gamma_2f \left[ \frac{1}{x-\lambda}-\frac{\langle (x-\lambda)^{-1},\phi\rangle}{D(\lambda)}\frac{\psi}{x-\lambda}\right] , \label{eq:7}\ee
It follows upon calculating the relevant integrals that
\be \Gamma_1f = \left[ \mbox{sign}(\Im\lambda) \pi i +\frac{\langle (x-\lambda)^{-1},\overline{\psi}\rangle
 \langle (x-\lambda)^{-1},\phi\rangle }{D(\lambda)} \right]\Gamma_2 f, \label{eq:8}\ee
and so
\be M_B(\lambda) = \left[ \mbox{sign}(\Im\lambda) \pi i +\frac{\langle (x-\lambda)^{-1},\overline{\psi}\rangle
 \langle (x-\lambda)^{-1},\phi\rangle }{D(\lambda)} - B\right]^{-1}. \label{eq:9}
\ee
\begin{rem}\label{remark:2}
If $D(\lambda)$ is nonzero then a local unique continuation
property holds:
\be f \in \ker(\At^*-\lambda)\cap\ker(\Gamma_1)\cap\ker(\Gamma_2) = 0
\;\;\; \Longrightarrow \;\;\; f = 0. \label{eq:luc}\ee
To see this observe that from (\ref{eq:7}) we see that $\Gamma_2f=0$
implies $f=0$, giving unique continuation a fortiori.
\end{rem}
\begin{rem}\label{remark:3}
Generically, the $M$-function $M_B(\lambda)$ `sees' the whole essential spectrum:
the term $\mbox{sign}(\Im(\lambda)\pi i)$ has a discontinuity across
the real axis which one cannot expect to be cancelled by the
other terms, except possibly on a set of measure zero.
\end{rem}
\begin{ex}\label{example:1}
If $\phi$ and $\psi$ both lie in the Hardy space $H^2_+$ 
(see Koosis \cite{Koosis} for definitions and properties of Hardy spaces)
then the inner product $\langle (x-\lambda)^{-1},\phi\rangle$ is zero for $\Im\lambda>0$ and the
inner product $\langle (x-\lambda)^{-1},\overline{\psi}\rangle$ is zero for
$\Im\lambda<0$. In this case $M_B(\lambda)$ has no poles and is given by
\[ M_B(\lambda) = (\mbox{sign}(\Im\lambda)\pi i - B)^{-1}. \]
If $B=\pi i$ then the entire upper half plane is filled with eigenvalues of
the operator $\At^*|_{\ker(\Gamma_1-B\Gamma_2)}$; if $B=-\pi i$ then it is the
lower half plane which is entirely filled with eigenvalues. 
\end{ex}

\begin{ex}\label{example:2}
We construct an example with a  particularly interesting property:
 an eigenvalue which is not a pole of the $M$-function.

Consider the case where
$\phi$ and $\psi$ both lie in $H^2_+$; fix $\lambda_0$ and, by choice
of $\phi$ and $\psi$, arrange that $D(\lambda_0)=0$. Avoid the pathological
cases where eigenvalues fill the entire upper or lower half planes by choosing
$B=0$; we have
\[ M_0(\lambda) = \frac{1}{\pi i} \mbox{sign}(\Im\lambda). \]
Consider the function
\[ u(x) = \frac{\psi(x)}{x-\lambda_0}. \]
Since $D(\lambda_0)=0$ it follows that $\langle u,\phi\rangle = -1$.
Moreover it is easy to check that $\Gamma_2u = c_{u} = 0$. It
is now easy to check that $u$ satisfies
\[ xu(x) + \langle u,\phi\rangle\psi = \frac{\lambda_0 \psi}{x-\lambda_0} \]
and so $u$ is an eigenfunction of $\left.\At^*\right|_{\ker(\Gamma_2)}$
with eigenvalue $\lambda_0$. However $\lambda_0$ is not a pole
of $M_0^{-1}(\lambda)$, in apparent contradiction to the results 
in \cite{MM02} and \cite{kn:BMNW} mentioned at the beginning of this section. 

Which hypotheses have failed?

If  $\Im\lambda_0<0$ then observe that
$\Gamma_1 u = \langle (x-\lambda_0)^{-1},\overline{\psi}\rangle = 0$, 
so the eigenfunction $u$ belongs to the domain of the minimal operator $A$, and hence 
to the domain of every extension: thus unique continuation fails, so there is
no contradiction to the theorems in \cite{kn:BMNW}. The failure
of unique continuation implies that there is no extension of $A$ for
which $\lambda_0$ lies in the resolvent set, and so there is no
contradiction to the results in \cite{MM02} either.

If  $\Im\lambda_0>0$ then although $\lambda_0$ is no longer an eigenvalue for
every extension, it nevertheless lies in the spectrum of every extension.
To see this, attempt to solve
\[ (x-\lambda)u-\Gamma_1 u + \langle u,\phi\rangle\psi = f, \]
\[ (\Gamma_1-C\Gamma_2)u=0, \]
with $\Im\lambda>0$.
Taking the inner products of both sides and remembering that $\langle (x-\lambda)^{-1},
\phi\rangle=0$ in the upper half plane we obtain
\[ \langle u, \phi \rangle = \langle \frac{f}{x-\lambda},\phi\rangle
 - \langle u,\phi\rangle \langle \frac{\psi}{x-\lambda},\phi \rangle. \]
At $\lambda=\lambda_0$ we have $\langle \frac{\psi}{x-\lambda_0},\phi \rangle = -1$ since
$D(\lambda_0)=0$ and so we obtain
\be \langle \frac{f}{x-\lambda_0},\phi\rangle = 0. \label{eq:10} \ee
Thus the problem can only be solved for $f$ satisfying the condition (\ref{eq:10})
and so $\lambda_0$ lies in the spectrum of every extension of $\At^*$. This gives
a further reason why we would not expect $\lambda_0$ to be a pole of any $M$-function.
\end{ex}

\begin{ex}\label{example:3}
In the case $\phi = \psi\in H^2_+$ the operators $\At^*|_{ker(\Gamma_1-B\Gamma_2)}$
are selfadjoint for real $B$. The functions $M_B(\lambda)$ still cannot `see' $\phi$
and $\psi$, however, being given by
\[ M_B(\lambda) = (\mbox{sign}(\Im\lambda)\pi i - B)^{-1}. \]
Any eigenvalues of the operator will obviously be real and will be imbedded in the 
essential spectrum. If $\lambda_0\in\Rr$ and $\psi(\lambda_0)=0$ and
\[ \int_{\Rr}\frac{|\psi(x)|^2}{x-\lambda_0}dx = -1, \]
which can always be arranged, then $\lambda_0$ will be an eigenvalue with eigenfunction
$\psi/(x-\lambda_0)$. The operator will not be unitarily equivalent to the unperturbed
operator, which has no eigenvalues. This is not surprising as the eigenfunction here
belongs to the minimal operator, which therefore fails to be completely non-selfadjoint.

There is therefore no contradiction to the results of Kre\u{\i}n, Langer and Textorius \cite{KL73,KL77,LT77} and Ryzhov \cite{Ryzhov} which
state that if the minimal operator is completely non-selfadjoint then the maximal
operator is determined up to unitary equivalence by the $M$-function.
\end{ex}

\end{document}